\documentclass[a4paper,10pt]{amsart}
\usepackage[english]{babel}
\usepackage{amsmath,amssymb,amsthm,enumerate}
\usepackage{multicol}
\usepackage[titletoc]{appendix}
\usepackage{wasysym}

\usepackage[pdftex]{color}
\usepackage[bookmarks=true,hyperindex,pdftex,colorlinks, citecolor=blue,linkcolor=blue, urlcolor=blue]{hyperref}
\usepackage[dvipsnames]{xcolor}

\usepackage{graphicx}

\usepackage[normalem]{ulem}

\theoremstyle{plain}
\newtheorem{theorem}{Theorem}[section]
\newtheorem{corollary}[theorem]{Corollary}
\newtheorem{proposition}[theorem]{Proposition}
\newtheorem{lemma}[theorem]{Lemma}

\theoremstyle{definition}

\newtheorem{example}[theorem]{Example}
\newtheorem{problem}[theorem]{Problem}
\newtheorem{definition}[theorem]{Definition}

 \DeclareMathOperator{\re}{Re\,}
 
 \DeclareMathOperator{\Id}{\mathrm{Id}}
 \DeclareMathOperator{\dist}{dist\,}

\DeclareMathOperator{\aconv}{aconv}

\newcommand{\K}{\mathbb{K}}

\newcommand{\C}{\mathbb{C}}
\newcommand{\R}{\mathbb{R}}
\newcommand{\N}{\mathbb{N}}

\newcommand{\norm}[1]{\ensuremath{\lVert#1\rVert}}

\newcommand{\ds}{\displaystyle}
\newcommand{\eps}{\varepsilon}

\renewcommand{\geq}{\geqslant}
\renewcommand{\leq}{\leqslant}

\begin{document}
\title{On the pointwise Bishop--Phelps--Bollob\'as property for operators}
\dedicatory{Dedicated to the memory of Joseph Diestel}

\author[Dantas]{Sheldon Dantas}
\address[Dantas]{Department of Mathematics, Faculty of Electrical Engineering, Czech Technical University in Prague, Technick\'a 2, 166 27 Prague 6, Czech Republic \newline
\href{http://orcid.org/0000-0001-8117-3760}{ORCID: \texttt{0000-0001-8117-3760}  }}
\email{\texttt{sheldon.dantas@gmail.com}}

\author[Kadets]{Vladimir Kadets}
\address[Kadets]{School of Mathematics and Computer Sciences \\
V. N. Karazin Kharkiv National University \\ pl.~Svobody~4 \\
61022~Kharkiv \\ Ukraine
\newline
\href{http://orcid.org/0000-0002-5606-2679}{ORCID: \texttt{0000-0002-5606-2679} }
}
\email{\texttt{v.kateds@karazin.ua}}

\author[Kim]{Sun Kwang Kim}
\address[Kim]{Department of Mathematics, Chungbuk National University, 1 Chungdae-ro, Seowon-Gu, Cheongju,
Chungbuk 28644, Republic of Korea \newline
\href{http://orcid.org/0000-0002-9402-2002}{ORCID: \texttt{0000-0002-9402-2002}  }}
\email{\texttt{skk@chungbuk.ac.kr}}

\author[Lee]{Han Ju Lee}
\address[Lee]{Department of Mathematics Education, Dongguk University - Seoul, 04620 (Seoul), Republic of Korea \newline
\href{http://orcid.org/0000-0001-9523-2987}{ORCID: \texttt{0000-0001-9523-2987}  }
}
\email{\texttt{hanjulee@dongguk.edu}}

\author[Mart\'{\i}n]{Miguel Mart\'{\i}n}
\address[Mart\'{\i}n]{Departamento de An\'{a}lisis Matem\'{a}tico, Facultad de
 Ciencias, Universidad de Granada, 18071 Granada, Spain \newline
\href{http://orcid.org/0000-0003-4502-798X}{ORCID: \texttt{0000-0003-4502-798X} }
 }
\email{\texttt{mmartins@ugr.es}}

\begin{abstract}
We study approximation of operators between Banach spaces $X$ and $Y$ that nearly attain their norms in a given point by operators that attain their norms at the same point. When such approximations exist, we say that the pair $(X, Y)$  has the  pointwise Bishop-Phelps-Bollob\'as property (pointwise BPB property for short). In this paper we mostly concentrate on those $X$, called universal pointwise BPB domain spaces, such that $(X, Y)$ possesses pointwise BPB property for every $Y$, and on those $Y$, called universal pointwise BPB range spaces, such that $(X, Y)$ enjoys pointwise BPB property for every uniformly smooth $X$. We show that every universal pointwise BPB domain space is uniformly convex and that $L_p(\mu)$ spaces fail to have this property when $p>2$. For universal pointwise BPB range space, we show that every simultaneously uniformly convex and uniformly smooth Banach space fails it if its dimension is greater than one. We also discuss a version of the pointwise BPB property for compact operators.
\end{abstract}

\date{September 18th, 2017. Revised June 28th, 2018.}

\thanks{The first author was supported by the project OPVVV CAAS CZ.02.1.01/0.0/0.0/16\_019/0000778. The research of the second author is done in frames of Ukrainian Ministry of Science and Education Research Program 0118U002036, and it was partially supported by the Spanish MINECO/FEDER projects MTM2015-65020-P and MTM2017-83262-C2-2-P. Third author was partially supported by Basic Science Research Program through the National Research Foundation of Korea(NRF) funded by the Ministry of Education, Science and Technology (NRF-2017R1C1B1002928). Fourth author was partially supported by Basic Science Research Program through the National Research Foundation of Korea (NRF) funded by the Ministry of Education, Science and Technology (NRF-2016R1D1A1B03934771). Fifth author partially supported by Spanish MINECO/FEDER grant MTM2015-65020-P}


\maketitle

\thispagestyle{plain}

\section{Introduction and Preliminaries}

Let $X$ and $Y$ be Banach spaces over the field $\K$ ($\K=\R$ or $\K=\C$). We denote by $\mathcal{L}(X,Y)$ the Banach space of all bounded linear operators from $X$ into $Y$ and by $\mathcal{K}(X,Y)$ its subspace of all compact linear operators. By $B_X$ and $S_X$, we denote the closed unit ball and the unit sphere of $X$, respectively. Our notation is standard and we will include a short list of notation and terminology in subsection \ref{subsection-notation} at the end of this introduction.

Motivated by the improvement given by Bollob\'{a}s of the classical Bishop-Phelps theorem on the denseness of norm-attaining functionals, Acosta, Aron, Garc\'{\i}a and Maestre introduced in 2008 \cite{AAGM} the so-called Bishop-Phelps-Bollob\'{a}s property as follows. A pair $(X,Y)$ of Banach spaces has the \emph{Bishop-Phelps-Bollob\'{a}s property} ({\it BPBp}, for short) if for every $\eps>0$, there exists $\eta(\eps)>0$ such that whenever $T\in \mathcal{L}(X,Y)$ with $\|T\|=1$ and $x_0\in S_X$ satisfy $\|Tx_0\|>1-\eta(\eps)$, there are $S\in \mathcal{L}(X,Y)$ with $\|S\| = 1$ and $x\in S_X$ such that $\|S\|=\|Sx\|=1$, $\|x_0-x\|<\eps$, and $\|S-T\|<\eps$. Let us observe that, trivially, if a pair $(X,Y)$ has the BPBp, then the set of all norm attaining operators from $X$ into $Y$ is dense in $\mathcal{L}(X,Y)$ (recall that $T\in \mathcal{L}(X,Y)$ is \emph{norm attaining} if there is $x\in S_X$ such that $\|T\|=\|Tx\|$). With this definition, Bollob\'{a}s' extension of the Bishop-Phelps theorem, nowadays known as the Bishop-Phelps-Bollob\'{a}s theorem, just says that the pair $(X,\K)$ has the BPBp for every Banach space $X$. Among many other results, it is known that a pair $(X,Y)$ has the BPBp when $X$ is  uniformly convex (\cite[Theorem 2.2]{ABGM} and \cite[Theorem 3.1]{KL}) or when $Y$ has Lindenstrauss property $\beta$ \cite[Theorem 2.2]{AAGM}, in particular, if $Y$ is a closed subspace of $\ell_\infty$ containing $c_0$ or if $Y$ is a finite-dimensional polyhedral space. On the other hand, for $\ell_1^2$, the $2$-dimensional space with the $\ell_1$-norm, there are Banach spaces $Y$ such that the pair $(\ell_1^2,Y)$ fails the BPBp (even though all elements in $\mathcal{L}(\ell_1^2,Y)$ attain their norms) \cite{AAGM}. Moreover, such $Y$ can be found in a way that for every Banach space $X$, the set of all norm attaining operators is dense in $\mathcal{L}(X,Y)$ \cite[Theorem 4.2]{ACKLM}. For more results and background on the BPBp, we refer the reader to the recent papers \cite{AcostaBJM2016, AMS, CGKS, Cho-Choi} and the references therein.

In the very recent paper \cite{DKL}, the following stronger version of the Bishop-Phelps-Bollob\'{a}s property was introduced (with the name of Bishop-Phelps-Bollob\'as point property).

\begin{definition}[\textrm{\cite[Definition 1.2]{DKL}}] We say that a pair $(X, Y)$ of Banach spaces has the \emph{pointwise Bishop-Phelps-Bollob\'as property} (\emph{pointwise BPB property}, for short) if given $\eps > 0$, there exists $\tilde\eta(\eps) > 0$ such that whenever $T \in \mathcal{L}(X, Y)$ with $\|T\| = 1$ and $x_0 \in S_X$ satisfy
	\begin{equation*}
	\|T(x_0)\| > 1 - \tilde\eta(\eps),	
	\end{equation*}
there is $S \in \mathcal{L}(X, Y)$ with $\|S\| = 1$ such that
\begin{equation*}
\|S(x_0)\| = 1 \qquad \text{and} \qquad \|S - T\| < \eps.	
\end{equation*}
\end{definition}

The difference between the BPBp and this new property is that, in the last one, the point $x_0$ does not move and the new operator $S$ attains its norm at it. Also observe that, by an easy change of parameters, we may consider $T \in \mathcal{L}(X, Y)$ with $\|T\| \leq 1$ in the above definition, and we will use this fact without any explicit reference. Nevertheless, the same trick does not work for the point $x_0$ (as the operator $S$ attains its norm at $x_0$), so we have to consider $x_0\in S_X$.

Another useful definitions, motivated by the corresponding ones for the BPBp given in \cite{ACKLM}, are the following.

\begin{definition}
A Banach space $X$ is said to be a \emph{universal pointwise BPB domain space} if $(X, Y)$ has the pointwise BPB property for every Banach space $Y$. A Banach space $Y$ is said to be a \emph{universal pointwise BPB range space} if $(X, Y)$ has the pointwise BPB property for every uniformly smooth Banach space $X$.
\end{definition}

As we already mentioned, the pointwise BPB property was introduced in \cite{DKL} where, among other results, it was proved that if $(X, Y)$ has the pointwise BPB property then $X$ must be uniformly smooth, that Hilbert spaces are universal pointwise BPB domain spaces, and that the class of universal pointwise BPB range spaces includes uniform algebras (in particular, $C(K)$-spaces) and those spaces with Lindenstrauss property $\beta$ (in particular, closed subspaces of $\ell_\infty$ containing the canonical copy $c_0$ or finite-dimensional polyhedral spaces).

As it was done for the BPBp in \cite{ACKLM}, it is useful to introduce the modulus of the pointwise BPB property for a pair of Banach spaces as follows.

\begin{definition}
Let $X, Y$ be Banach spaces, and let $\eps \in (0, 1)$. The \emph{modulus of the pointwise BPB property} for the pair $(X, Y)$, denoted by  $\tilde\eta(X, Y)(\eps)$, is defined as the supremum of the set consisting of $0$ and all those $\rho > 0$ such that for all $T \in S_{\mathcal{L}(X, Y)}$ and $x_0 \in S_X$ with $\|T(x_0)\| > 1 - \rho$, there is $S \in S_{\mathcal{L}(X, Y)}$ such that $\|S(x_0)\| = 1$ and $\|S - T\| < \eps$. Equivalently,
\begin{equation*}
\tilde\eta(X, Y)(\eps) = \inf \left\{ 1 - \|Tx\|\colon \ x \in S_X, \ T \in S_{\mathcal{L}(X, Y)},  \ \dist\bigl(T,\{ S\in \mathcal{L}(X, Y)\colon \|S\|=1= \|Sx\|\}\bigr) \geq \eps \right\}.
\end{equation*}
\end{definition}

It is clear from the definition that a pair $(X, Y)$ has the pointwise BPB property if and only if $\tilde\eta(X,Y)(\eps)>0$ for all $\eps \in (0, 1)$.

Our aim in this paper is to continue the study of the pointwise Bishop-Phelps-Bollob\'{a}s property initiated in \cite{DKL}.

Let us now give an account of the main results of this manuscript. In section \ref{sec:stability}, we provide some stability results for the pointwise BPB property. For instance, if $(X,Y)$ has the pointwise BPB property, then so do $(X_1,Y)$ if $X_1$ is one-complemented in $X$ and $(X,Y_1)$ if $Y_1$ is an absolute summand of $Y$ (in particular, if it is an $\ell_p$-summand). Consequently, given a universal pointwise BPB domain space $X$, there is a common function which makes all the pairs $(X,Y)$ to have the pointwise BPB property with this function for every Banach space $Y$. On the other hand, we show that such  function does not exist for any universal pointwise BPB range space. Section \ref{sect:domain} is devoted to universal pointwise BPB domain spaces. We show that they are uniformly convex and we give a control on their modulus of convexity which allows us to show that $L_p(\mu)$ spaces are not universal pointwise BPB domain spaces for $p>2$. The class of universal pointwise BPB range spaces is studied in section \ref{sec:range-spaces} where it is shown that it contains all spaces with ACK$_\rho$-structure (see Definition \ref{def-ack-str}) introduced in the very recent paper \cite{CGKS}; examples of Banach spaces with this structure are, for instance, all those with property $\beta$, uniform algebras, and some vector-valued function spaces. We also show that a universal pointwise BPB range space of dimension greater than 1 cannot be simultaneously uniformly convex and uniformly smooth, and that the pointwise BPB property is not stable by infinite $c_0$- and $\ell_p$-sums of the range space for $1\leq p \leq \infty$. Section \ref{sect:compact} is devoted to the pointwise BPB property for compact operators (the analogous definition considering only compact operators). There are many results for the pointwise BPB property which are also true for the compact operators version, and we list all of these in this section. Moreover, some results for the BPBp for compact operators from \cite{DGMM} are adapted to the pointwise version for this class of operators. Finally, in the last section, we compile some open problems.

We would like to dedicate this paper to the memory of Joe Diestel, who passed away recently. Joe wrote some books and textbooks which became classic in the study on Banach spaces theory. Among them, we may empathize the one entitled ``Geometry of Banach spaces'' \cite{Diestel}. This classical book gives a deep insight into the subject of this paper. In particular, it contains the proof of the Bishop-Phelps theorem and the properties of uniformly convex and uniformly smooth Banach spaces which are essential for us.

\subsection{Notation and terminology}\label{subsection-notation}
A Banach space $X$ (or its norm) is said to be \emph{smooth} if the norm is G\^{a}teaux differentiable at every non-zero point of $X$ or, equivalently, if for every $x\in X\setminus \{0\}$ there is only one $x^*\in S_{X^*}$ such that $x^*(x)=\|x\|$. In this case, we may consider the \emph{duality mapping} $J_X:X\setminus\{0\}\longrightarrow S_{X^*}$ defined by the formula $\langle J_X(x),x\rangle = \|x\|$ for every $x\in X$. The \emph{modulus of smoothness} of a Banach space $X$ is defined as follows:
\begin{equation*}
\rho_X (t) := \sup \left\{\frac{\|x + y\|+\|x - y\|-2}{2} \colon \|x\| = 1 \ \text{and} \ \|y\| \leq t  \right\} \qquad (0<t\leq 1).	
\end{equation*}
The space $X$ (or its norm) is said to be \emph{uniformly smooth} if $\displaystyle\lim\limits_{t\to 0}\frac{\rho_X(t)}{t}=0$ (equivalently, if the norm is Fr\'{e}chet differentiable on $S_X$ and the duality mapping is uniformly continuous in $S_X$).
The \emph{modulus of convexity} of a Banach space $X$ is given by
$$
\delta_X(\eps)=\inf\left\{1-\left\|\frac{x+y}{2}\right\|\colon x,y\in B_X,\, \|x-y\|\geq \eps\right\} \qquad (0<\eps\leq 2).
$$
The space $X$ (or its norm) is said to be \emph{uniformly convex} if $\delta_X(\eps)>0$ for all $\eps\in (0,2]$. We refer the reader to \cite[\S 9]{FHHMPZ} and \cite{Diestel} for background on uniform convexity and uniform smoothness.

Given $1\leq p \leq \infty$, a positive measure $\mu$, and a Banach space $X$, $L_p(\mu,X)$ is the Banach space of those strongly $\mu$-measurable functions which are $p$-integrable if $1\leq p<\infty$ or essentially bounded for $p=\infty$, endowed with the corresponding natural norm in each case. We write $\ell_p^n(X)=L_p(\nu,X)$, where $\nu$ is the counting measure on $\{1,\ldots,n\}$, that is, the product space $X^n$ endowed with the $p$-norm. When $X=\K$, we just write $\ell_p^n$. We also use the notation $\ell_p(X)=L_p(\nu,X)$, where $\nu$ is the counting measure on $\N$. Finally, given a family $\{Y_j\colon j\in J\}$ of Banach spaces, $\Bigl[ \bigoplus_{j \in J} Y_j \Bigr]_{c_0}$ denotes its $c_0$-sum and $\Bigl[ \bigoplus_{j \in J} Y_j \Bigr]_{\ell_p}$ denotes its $\ell_p$-sum for $1 \leq p \leq \infty$.

\section{Some stability results}\label{sec:stability}

Our aim in this section is to present some stability results for the pointwise BPB property. Most of the results are generalizations of the known corresponding results for the BPBp. But in the case of the domain spaces, the new result is much more general. Recall that a subspace $X_1$ of a Banach space $X$ is said to be \emph{one-complemented} if $X_1$ is the range of a norm-one projection on $X$.

\begin{proposition} \label{prop:abssumdomain:pointwise BPB property}
Let $X$ be a uniformly smooth Banach space, let $Y$ be an arbitrary Banach space, and let $X_1$ be a one-complemented subspace of $X$. If $(X,Y)$ has the pointwise BPB property, then so does the pair $(X_1,Y)$. Moreover, $\tilde\eta(X,Y)(\eps)\leq \tilde\eta(X_1,Y)(\eps)$ for every $\eps\in (0,1)$.
\end{proposition}

\begin{proof}
We write $i:X_1\longrightarrow X$ for the inclusion and consider the norm-one operator $P:X\longrightarrow X_1$ such that $P\circ i=\Id_{X_1}$ (the existence of such $P$ is equivalent to the fact that $X_1$ is one-complemented in $X$). Suppose that $(X,Y)$ has the pointwise BPB property with a function $\eps\longmapsto \tilde\eta(\eps)$. Fix $\eps\in (0,1)$ and consider $T\in L(X_1,Y)$ with $\|T\|=1$ and $x_0\in S_{X_1}$ such that $\|Tx_0\|> 1 - \tilde\eta(\eps)$. Write $\widetilde{T}=T\circ P\in \mathcal{L}(X,Y)$ and observe that $\|\widetilde{T}\|\leq 1$, $\|\widetilde{T}\circ i (x_0)\|=\|Tx_0\|>1-\tilde\eta(\eps)$, so there exists $\widetilde{S}\in \mathcal{L}(X,Y)$ with $\|\widetilde{S}\|=\|\widetilde{S}\circ i(x_0)\|=1$ and $\|\widetilde{S}-\widetilde{T}\|<\eps$. Now, consider $S=\widetilde{S}\circ i\in L(X_1,Y)$ and observe that $\|S\|=\|Sx_0\|=1$ and that $\|S-T\|\leq \|\widetilde{S}-\widetilde{T}\|<\eps$.
\end{proof}

As an immediate consequence of the above result, one-complemented subspaces of universal pointwise BPB domain spaces are also universal pointwise BPB domain spaces.

\begin{corollary}\label{corollary-universalpointwise BPB property-one-complemented}
Let $X$ be a uniformly smooth Banach space and let $X_1$ be a one-complemented subspace of $X$. If $X$ is a universal pointwise BPB domain space, then so is $X_1$.
\end{corollary}

Let us comment that we do not know whether the analogous results for the BPBp are true. We send the reader to the very recent paper \cite{CDJM} to see some particular cases in which this is the case. We also do not know whether the density of norm attaining operators passes to one complemented subspaces of the domain space.

Another immediate consequence of Proposition \ref{prop:abssumdomain:pointwise BPB property} is the following.

\begin{corollary} Let $\{X_i\colon i \in I \}$ be a family of Banach spaces and let $Y$ be a Banach space. Put $X = \left[ \bigoplus_{i \in I} X_i \right]_{\ell_p}$ for $1 < p < \infty$. If the pair $(X, Y)$ has the pointwise BPB property, then all the pairs $(X_i, Y)$ have the pointwise BPB property. Moreover, $\tilde\eta(X,Y)(\eps)\leq \tilde\eta(X_i,Y)(\eps)$ for every $\eps\in (0,1)$ and every $i\in I$.
\end{corollary}	

Other easy consequence of Proposition \ref{prop:abssumdomain:pointwise BPB property} is the following one, which follows immediately from the fact that a Banach space $X$ is one-complemented in every $L_p(\mu,X)$ space (via conditional expectation).

\begin{corollary}
Let $X$ be a uniformly smooth Banach space, let $Y$ be an arbitrary Banach space, let $1 < p < \infty$ be given, and let $\mu$ be a positive measure. If $(L_p(\mu,X),Y)$ has the pointwise BPB property, then so does $(X,Y)$.
\end{corollary}


For range spaces, we do not know whether a result as in Proposition \ref{prop:abssumdomain:pointwise BPB property} exists, but we may get a similar result for subspaces which are absolute summands. Let us now recall some basic facts about absolute sums. An \emph{absolute norm} is a norm $| \cdot  |_a$ on $\R^2$ such that $|(1, 0)|_a = |(0, 1)|_a = 1$ and $|(s, t)|_a = |(|s|, |t|)|_a$ for every $s, t \in \R$. Given two Banach spaces $Y$ and $W$ and an absolute norm $|  \cdot |_a$, the \emph{absolute sum} of $Y$ and $W$ with respect to $|  \cdot  |_a$, denoted by $Y \oplus_a W$, is the Banach space $Y \times W$ endowed with the norm
\begin{equation*}
\|(y, w)\|_a = |(\|y\|, \|w\|)|_a \qquad (y \in Y, w \in W)
\end{equation*}
Examples of absolute sums are the $\ell_p$-sums for $1 \leq p \leq \infty$ associated to the $\ell_p$-norm in $\R^2$. If $X=Y\oplus_a Z$, we also say that $Y$ and $Z$ are \emph{absolute summands} of $X$.

\begin{proposition} \label{prop:abssumrange:pointwise BPB property} Let $X$ be a uniformly smooth Banach space, let $Y_1$ and $Y_2$ be arbitrary Banach spaces and let $Y=Y_1\oplus_a Y_2$ for an absolute norm $|  \cdot  |_a$. If the pair $(X, Y)$ has the pointwise BPB property, then the pairs $(X, Y_1)$ and $(X, Y_2)$ have the pointwise BPB property. Moreover, $\tilde\eta(X,Y_i)(\eps)\geq \tilde\eta(X,Y)(\eps/3)$ for every $\eps\in (0,1)$ and $i=1,2$.
\end{proposition}

The proof is an adaptation of corresponding one for the BPBp given in \cite{CDJM}. On the other hand, when the absolute sum is $\ell_\infty$- or $\ell_1$-sum, then it is not needed to divide $\eps$ by $3$ in the above result since we may adapt the arguments given in \cite[Propositions 2.4 and 2.7]{ACKLM}.

In particular, we easily get the following consequence.
	
\begin{corollary}\label{corollary:sums-range} Let $X$ be a uniformly smooth Banach space and let $\{Y_j\colon j \in J \}$ be a family of Banach spaces. Put $Y = \Bigl[ \bigoplus_{j \in J} Y_j \Bigr]_{c_0}$ or $Y = \Bigl[ \bigoplus_{j \in J} Y_j \Bigr]_{\ell_p}$ for $1 \leq p \leq \infty$. If the pair $(X, Y)$ has the pointwise BPB property for some $\tilde\eta(\eps)$, then all the pairs $(X, Y_j)$ have the pointwise BPB property. Moreover, $\tilde\eta(X,Y_j)(\eps)\geq \tilde\eta(X,Y)(\eps/3)$ for every $\eps\in (0,1)$ and $j \in J$.
\end{corollary}	

Let us comment that for $c_0$- and $\ell_\infty$-sums, the converse of the above result also holds \cite[Proposition 2.9]{DKL}. Indeed, if $X$ is a uniformly smooth Banach space and $\{Y_j\colon j \in J \}$ is a family of Banach spaces
such that $\inf\{\tilde\eta(X,Y_j)(\eps)\colon j\in J\}>0$ for every $\eps\in(0,1)$, then the pairs $\left(X,\left[ \bigoplus_{j \in J} Y_j \right]_{c_0}\right)$ and  $\left(X,\left[ \bigoplus_{j \in J} Y_j \right]_{\ell_\infty}\right)$ have the pointwise BPB property. We will show in section \ref{sec:range-spaces} that the analogous result fails for $\ell_p$-sums with $p\geq 2$.

Another stability result in the same line is the following one whose proof is a routine adaptation of the one in \cite[Proposition~2.8]{ACKLM}.

\begin{proposition}\label{prop:X-CKY}
Let $X$ be a uniformly smooth Banach space, let $Y$ be an arbitrary Banach space, and let $K$ be a compact Hausdorff topological space. If $(X,C(K,Y))$ has the pointwise BPB property, then $(X,Y)$ has the pointwise BPB property. Moreover, $\tilde\eta(X,Y)(\eps)\geq \tilde\eta(X,C(K,Y))(\eps)$ for every $\eps\in (0,1)$.
\end{proposition}

As a consequence of Proposition \ref{prop:abssumrange:pointwise BPB property}, we get the following result, which is analogous to \cite[Corollary 2.2]{ACKLM} and which can actually be deduced from already mentioned results of \cite{DKL}. This will be very important in the next section.

\begin{corollary}\label{corollary:universal-eta}
If a Banach space $X$ is a universal pointwise BPB domain space, then there is a function $\tilde\eta_X: (0, 1) \longrightarrow \R^+$ such that $\tilde\eta(X, Z)(\eps)\geq \tilde\eta_X(\eps)$ for every $\eps\in (0,1)$ and every Banach space $Z$.
\end{corollary}

\begin{proof}
Assume that such $\tilde\eta_X$ does not exist. Then, for some $\eps_0 > 0$, there exists a sequence of  Banach spaces $\{Y_n\colon n\in \N \}$ such that all the pairs $(X, Y_n)$ have the pointwise BPB property and $\tilde\eta(X, Y_n)(\eps_0) \longrightarrow 0$ when $n \longrightarrow \infty$. But if we consider $Y = \left[\bigoplus_{n \in \N} Y_n \right]_{\ell_p}$ for $1\leq p\leq \infty$, then the pair $(X, Y)$ has the pointwise BPB property by hypothesis, and Proposition \ref{prop:abssumdomain:pointwise BPB property} gives that $\tilde\eta(X, Y_n)(\eps_0) \geq \tilde\eta(X, Y)(\eps_0/3) > 0$, which is a contradiction.
\end{proof}

One may wonder if the analogous result for universal pointwise BPB range spaces is also true. If we try to repeat the proof, we quickly get into trouble as the $\ell_p$-sum of a family of uniformly smooth Banach spaces needs not to be uniformly smooth. Actually, we are going to prove that there is no analogous result for range spaces. We start by proving a quantitative form of the fact that Banach spaces $X$ for which $(X,\K)$ has the pointwise BPB property are uniformly smooth.

\begin{proposition} \label{pointwise BPB property:modulussmooth} Suppose that the pair $(X, \K)$ has the pointwise BPB property with some function $\tilde\eta$. Then, given $\eps\in (0,1)$,
\begin{equation*}
\frac{\rho_X(t)}{t} <2\eps + 4 \sqrt{\tilde\eta(\eps)+\eps}	
\end{equation*}
for every $t < \frac{\tilde\eta(\eps)}{2}$.
\end{proposition}

\begin{proof} Let $\eps\in (0,1)$ be given and fix $0<t<\tilde\eta(\eps)/2$. Consider $x_0 \in S_X$ and $y_0\in X$ with $\|y_0\| \leq t$. Let $u^*,w^*\in S_{X^*}$ be such that $\re u^*(x_0 + y_0) = \|x_0 + y_0\|$ and $\re w^*(x_0 - y_0) = \|x_0 - y_0\|$. Then
\begin{equation*}
\re u^*(x_0) \geq \|x_0 + y_0\| - \|y_0\| > 1 - \tilde\eta(\eps).	
\end{equation*}
Analogously, $\re w^*(x_0) > 1 - \tilde\eta(\eps)$. Since $(X, \K)$ has the pointwise BPB property with $\tilde\eta$, there are $v_1^* \in S_{X^*}$ and $v_2^* \in S_{X^*}$ such that $|v_1^*(x_0)| = |v_2^*(x_0)| = 1$, $\|v_1^* - u^*\| < \eps$ and $\|v_2^* - w^*\| < \eps$. There are $\alpha,\beta\in \K$ with $|\alpha|=|\beta|=1$ such that $\alpha v_1^*(x_0)=\beta v_2^*(x_0)=1$ so, as $X$ is (uniformly) smooth \cite[Proposition 2.1]{DKL}, $\alpha v_1^*=\beta v_2^*=J_X(x_0)$. Since
$$
\re v_1^*(x_0) \geq \re u^*(x_0) - \|v_1^*-u^*\| > 1-\tilde\eta(\eps)-\eps,
$$
it follows that
$$
\re \frac{\alpha v_1^*(x_0) + v_1^*(x_0)}{2}\geq 1 - \frac{\tilde\eta(\eps)+\eps}{2}.
$$
Now, by the parallelogram law,
\begin{align*}
\frac{|\alpha-1|^2}{4} &  = 1 - \frac{|\alpha+1|^2}{4}=1 - \left|\frac{\alpha v_1^*(x_0)+ v_1^*(x_0)}{2} \right|^2 \\ & \leq 1 -\left(1-\frac{\tilde\eta(\eps)+\eps}{2}\right)^2 \leq \tilde\eta(\eps)+\eps.
\end{align*}
It follows that $|\alpha-1|\leq 2\sqrt{\tilde\eta(\eps)+\eps}$ and, analogously, $|\beta-1|\leq 2\sqrt{\tilde\eta(\eps)+\eps}$. So, since $\alpha v_1^* = \beta v_2^*$, we have that
\begin{equation*}
\|u^* - w^*\| \leq \|u^* - v_1^*\| + |\alpha-1| + \|v_2^* - w^*\| + |\beta-1| < 2 \eps  + 4 \sqrt{\tilde\eta(\eps)+\eps}.	
\end{equation*}	
Therefore,
\begin{align*}
\|x_0 + y_0\| + \|x_0 - y_0\| - 2 &= \re u^*(x_0 + y_0) + \re w^*	(x_0 - y_0) - 2 \\ &= \re u^*(x_0) + \re w^*(x_0) + \re (u^* - w^*)(y_0) - 2 \\
&\leq \re (u^* - w^*)(y_0) \leq \|u^* - w^*\|\|y_0\| < \bigl(2 \eps  + 4 \sqrt{\tilde\eta(\eps)+\eps}\bigr) t.\qedhere
\end{align*}
\end{proof}

Observe that, as $\tilde\eta(X,\K)(\eps)$ goes to $0$ when $\eps$ goes to $0$ (unless $X$ is one-dimensional), it follows that if $(X,\K)$ has the pointwise BPB property then $X$ is uniformly smooth with a control on the modulus of smoothness of $X$ which only depends on $\tilde\eta$. Since it is possible to construct uniformly smooth  Banach spaces whose  moduli of smoothness are as bad as we want (consider $\ell_n^2$ with $n\in \N$, for instance),  there is no universal function  $\tilde\eta$ such that all pairs $(X,\K)$ with $X$ being uniformly smooth  enjoy  the pointwise BPB property with $\tilde\eta$.

\begin{corollary}\label{corollary:no-universalfunctionforrange}
Let $Y$ be a Banach space. Then for every $\eps\in (0,1)$,
$$
\inf \bigl\{ \tilde\eta(X,Y)(\eps)\colon X \text{ is a uniformly smooth Banach space}\bigr\}=0.
$$
\end{corollary}

\begin{proof}
It is shown in \cite[Proposition 2.3]{DKL} that when a pair $(X,Y)$ has the pointwise BPB property with a function $\eps\longmapsto \tilde\eta(\eps)$, then the pair $(X,\K)$ has the pointwise BPB property with the function $\eps\longmapsto \tilde\eta(\eps/2)$. Therefore, if for a Banach space $Y$ the infimum in the statement is not zero, then there is a universal function $\eps\longmapsto \tilde\eta'(\eps)>0$ such that for every uniformly smooth Banach space $X$ the pair $(X,\K)$ has the pointwise BPB property with the function $\tilde\eta'$. Then, Proposition \ref{pointwise BPB property:modulussmooth} gives a control on the modulus of smoothness which is valid for all uniformly smooth Banach spaces $X$ and this is impossible.
\end{proof}

\section{Universal pointwise BPB domain spaces}\label{sect:domain}

Our main goal in this section is to prove that universal pointwise BPB domain spaces are uniformly convex, providing a control on their modulus of convexity. We will extract some important consequences of this.

\begin{theorem}\label{thm1}
If $X$ is a universal pointwise BPB domain space, then $X$ is uniformly convex. Moreover, we have that
\[
\delta_{X}(\eps) \geq C\,\eps^q\qquad (0<\eps<2)
\]
for suitable $2\leq q<\infty$ and $C>0$.
\end{theorem}

Prior to provide the proof of the theorem, we need three preliminary results. The first one is the following lemma which is useful to compute the modulus of convexity of certain types of renormings.

\begin{lemma}\label{lem1}
Let $(X, \|\cdot\|)$ be a Banach space. Suppose that there are an equivalent norm $||| \cdot |||$ on $X$,  $0<\delta\leq 1$, $2\leq q<\infty$, and $c>0$, such that $||| x||| \leq \|x\| \leq c|||x|||$ and
\[
\left( \frac{|||x+y|||}{2}\right)^q+\delta \left( \frac{|||x-y|||}{2}\right)^q \leq \frac{|||x|||^q+|||y|||^q}{2}
\]
for all $x, y\in X$. For each $n \in \N$, define on $X$ the equivalent norm
\begin{equation*}
\|x\|_n := \left( \|x\|^q + \frac{1}{n}|||x|||^q \right)^{1/q} \qquad (x \in X)
\end{equation*}
and let $X_n = (X, \| \cdot \|_n)$. Then
\[
\delta_{X_n}(\eps) \geq 1- \left(1-\frac{\delta \eps^q}{2^q(nc^q + 1)}\right)^{1/q}\geq \frac{\delta \eps^q}{q2^q(nc^q+1)}\qquad (0<\eps<2).
\]
\end{lemma}

\begin{proof}
Let $0<\eps<2$ be given and suppose that $\|x\|_n = \|y\|_n =1$ and $\|x-y\|_n \geq \eps$. We then have that
\begin{equation*}
\eps \leq \|x - y\|_n \leq
\left(c^q + \frac 1n\right)^{1/q}|||x - y|||.
\end{equation*}
This implies that $|||x - y||| \geq \dfrac{\eps}{(c^q+1/n)^{1/q}}$. Applying this, the convexity of the function $t\longmapsto t^q$ for $q\geq 2$, and the hypothesis, we get that
\begin{align*}
\left( \frac{\|x + y\|_n}{2} \right)^q &= \left(\frac{\|x+y\|}{2}\right)^q + \frac{1}{n}\left(\frac{||| x + y |||}{2}\right)^q \\
&\leq \frac{\|x\|^q+\|y\|^q}{2} + \frac{1}{n} \left(  \frac{|||x|||^q + |||y|||^q}{2} - \delta \left( \frac{|||x - y|||}{2}\right)^q\right)\\
&= 1- \frac{\delta}{2^qn}|||x-y|||^q \leq 1- \frac{\delta\eps^q}{2^q(nc^q+1)}.\qedhere
\end{align*}
\end{proof}

The second preliminary result allows us to compute the norm of the inverse of some  operators.

\begin{lemma}\label{lemma:norm_of_the_inverse}
Let $X$ be a vector space endowed with two equivalent complete norms $\|\cdot\|_1$ and $\|\cdot\|_2$ such that
\begin{equation} \label{eqeqeq1}
\|x\|_1 \leq  \|x\|_2  \qquad (x\in X)
\end{equation}
and consider the identity operator $\Id:(X,\|\cdot\|_1)\longrightarrow (X,\|\cdot\|_2)$. If $S:(X,\|\cdot\|_1)\longrightarrow (X,\|\cdot\|_2)$ satisfies that $\|S - \Id\|<1/4$, then $S$ is invertible and $\|S^{-1}\|<4/3$.
\end{lemma}

\begin{proof}
Let us denote  $X_1 = (X,\|\cdot\|_1)$,  $X_2 = (X,\|\cdot\|_2)$ and let $I_1: X_1 \longrightarrow X_1$, $I_2: X_2 \longrightarrow X_2$ be the corresponding identity operators. Denote also $S_2 = S \circ \Id^{-1}: X_2 \longrightarrow X_2 $. Condition \eqref{eqeqeq1} means that $\|\Id^{-1}\| \leq 1$,
so $\|S_2 - I_2\|  = \|(S - \Id)\Id^{-1}\| < 1/4$. Then, $S_2$ is invertible with
$S_2^{-1} = I_2 + \sum_{k=1}^\infty(I_2 - S_2)^k$. Consequently, we have $\|S_2^{-1}\| < 4/3$. To complete the proof it remains to remark that  $\|S^{-1}\| = \|\Id^{-1} S_2^{-1} \| \leq \|S_2^{-1}\| <4/3$.
\end{proof}

Finally, we provide a third preliminary result which will allow us to estimate the modulus of convexity of some Banach spaces.

\begin{lemma}\label{lemma-estimation-modulus-of-convexity}
Let $X$ be a smooth Banach space and let $\delta: (0,2)\longrightarrow \R^+$ be a function. Suppose that given $\eps \in (0, 2)$, $x, z \in S_X$ satisfying
\begin{equation*}
\re \langle J_{X}(x), z \rangle > 1 - \delta (\eps),	
\end{equation*}
one has $\|x - z\| < \frac{\eps}{2}$.
Then,
\begin{equation*}
\delta_{X} (\eps) \geq \frac{1}{2} \delta(\eps).
\end{equation*}
\end{lemma}

\begin{proof}
Let $z_1, z_2 \in S_{X}$ be such that $\|z_1 - z_2\| \geq \eps$. So, for all $x \in S_{X}$, we have
\begin{equation*}
\re \langle J_{X}(x), z_1 \rangle \leq 1 - \delta (\eps) \quad \text{or} \quad \re \langle J_{X}(x), z_2 \rangle \leq 1 - \delta(\eps).	
\end{equation*}
Otherwise, we would have $\|z_1 - x\| < \frac{\eps}{2}$ and  $\|z_2 - x\| < \frac{\eps}{2}$, so $\|z_1 - z_2\| < \eps$, a contradiction. Then, for every $x \in S_{X}$, we have that
\begin{equation*}
\re \left\langle J_{X}(x), \frac{z_1 + z_2}{2} \right\rangle \leq 1 - \frac{1}{2} \delta (\eps).	
\end{equation*}
Since $X$ is smooth and the above inequality holds for all $x \in S_{X}$, we get that
\begin{equation*}
\left\| \frac{z_1 + z_2}{2} \right\| \leq 1 - \frac{1}{2} \delta (\eps),	
\end{equation*}
which implies that
$\delta_{X} (\eps) \geq \frac{1}{2} \delta(\eps).$
\end{proof}

We are now ready to prove our main result.

\begin{proof}[Proof of Theorem~\ref{thm1}]
Let $\|\cdot\|$ be the norm of $X$ and let $\tilde\eta_X$ be the universal pointwise BPB function given by Corollary \ref{corollary:universal-eta}. First of all, $X$ is uniformly smooth (see \cite[Proposition 2.1]{DKL}) and so, it is supperreflexive. We may then use a well-known result by Pisier \cite{Pi} to get that $X$ admits an equivalent uniformly convex norm $||| \cdot |||$ for which there exist $\delta > 0$ and $2 \leq q < \infty$ such that
\begin{equation}\label{equation:BPB-domain-Pisier}
\left( \frac{|||x + y|||}{2} \right)^q + \delta \left( \frac{|||x-y|||}{2}\right)^q \leq \frac{|||x|||^q+|||y|||^q}{2} \qquad \bigl(x,y \in X\bigr).
\end{equation}
We may also assume that there exists $c>0$ such that $|||x||| \leq \|x\| \leq c |||x|||$ for every $x \in X$. For each $m \in \N$, we define an equivalent norm on $X$ by
\begin{equation*}
\|x\|_m := \left(\|x\|^q + \frac{1}{m}|||x|||^q \right)^{1/q} \qquad (x \in X),	
\end{equation*}
which satisfies that
\begin{equation} \label{eq4}
\|x\| \leq \|x\|_m \leq \left(1 + \frac{1}{m} \right)^{1/q}\|x\|
\end{equation}	
for all $x \in X$ and for all $m \in \N$.

Now, consider $m_0\in \N$ such that
\begin{equation*}
\left(\frac{m_0}{1 + m_0}\right)^{1/q} > 1 - \tilde\eta_X\left( \frac{1}{8}\right) \qquad \text{and} \qquad  \frac{1}{8} +  \left(\frac{m_0+1}{m_0}\right)^{1/q} - 1 < \frac{1}{4}.	
\end{equation*}
We claim that \emph{given $\eps \in (0, 2)$, $x, z \in S_X$ satisfying that
\begin{equation*}
\re \langle J_X(x), z \rangle > 1 - \delta_{X_{m_0}} \left(\frac{3}{8} \eps \right),	
\end{equation*}
one has that $\ds \|x - z\| < \frac{\eps}{2}$.} This finishes the proof of the theorem by just using Lemma \ref{lemma-estimation-modulus-of-convexity} to get that
$$
\delta_{X} (\eps) \geq \frac{1}{2} \delta_{X_{m_0}}\left(\frac{3}{8}\eps \right)
$$
for every $\eps\in (0,2)$, and then Lemma \ref{lem1} to get the desired estimation for $\delta_{X_{m_0}}$, and so for $\delta_X$.

Let us prove the claim. Define $\Id: (X, \| \cdot \|) \longrightarrow (X, ||| \cdot |||_{m_0})$ to be the identity map. By \eqref{eq4}, we get that $1 \leq \|\Id\| \leq \left(1 + \frac{1}{m_0} \right)^{1/q}$. Let $T=\frac{1}{\|\Id\|}\Id$. Fix now $x\in S_X$ and observe that
\begin{equation*}
\|Tx\|_{m_0} = \frac{\|x\|_{m_0}}{\|\Id\|} \geq \left(\frac{m_0}{1 + m_0}\right)^{1/q} > 1 - \tilde\eta_X \left(\frac{1}{8}\right).	
\end{equation*}
Since the pair $((X, \| \cdot \|), (X,\| \cdot \|_{m_0}))$ has the pointwise BPB property with  $\tilde\eta_X$, there exists $S: (X, \| \cdot \|) \longrightarrow (X, \|\cdot\|_{m_0})$ with $\|S\| = 1$ such that
\begin{equation*}
\|S(x)\|_{m_0} = 1 \qquad \text{and} \qquad \|S - T\| < \frac{1}{8}.	
\end{equation*}
So, we get that
\begin{equation*}
\|S - \Id\| \leq \|S - T\| + \|T - \Id\| = \|S - T\| + \|\Id\| - 1 \leq \frac{1}{8} +  \left(\frac{m_0+1}{m_0}\right)^{1/q} - 1  < \frac{1}{4}.
\end{equation*}
By Lemma \ref{lemma:norm_of_the_inverse}, we get that $S$ is invertible with $\|S^{-1}\|<4/3$. Now, take $y^* \in S_{(X, \|\cdot\|_{m_0})^*}$ to be such that $y^*(S(x)) = S^*y^*(x) = 1$. By the smoothness of $X$, we have that $S^*y^* = J_X(x)$. Finally, if $z \in S_X$ is such that
\begin{equation*}
\re \langle J_X(x), z \rangle > 1 - \delta_{X_{m_0}} \left( \frac{3}{8} \eps \right),	
\end{equation*}
then
\begin{equation*}
\re \langle y^*, Sz \rangle = \re \langle S^* y^*, z \rangle = \re \langle J_X(x), z \rangle > 1 - \delta_{X_{m_0}} \left( \frac{3}{8} \eps \right).	
\end{equation*}
Since $\langle y^*, Sx\rangle = 1$, we have that
\begin{equation*}
\left\| \frac{Sx + Sz}{2} \right\|_{m_0} \geq \re \left \langle y^*, \frac{Sx + Sz}{2} \right \rangle > 1 - \delta_{X_{m_0}} \left( \frac{3}{8} \eps \right).
\end{equation*}
This implies that $\ds \|Sx - Sz\|_{m_0} < \frac{3}{8} \eps$. Now, since $\ds \|S^{-1}\| < \frac{4}{3}$, we get that
\begin{equation*}\|z - x\| = \|S^{-1}(Sz) - S^{-1}(Sx)\| \leq \|S^{-1}\| \|Sx - Sz\|_{m_0} < \frac{\eps}{2}.
\end{equation*}
This proofs the claim.
\end{proof}

Let us observe that, in the above proof, if one starts with a Banach space $X$ which is isomorphic to a Hilbert space, then there is no need to use Pisier's result, as inequality \eqref{equation:BPB-domain-Pisier} with $q=2$ and $\delta=1$ follows immediately from the parallelogram law, so we get that $\delta_X(\eps)\geq C\eps^2$ for some $C>0$ in this case. Moreover, if one follows all the arguments giving in the proof of Theorem \ref{thm1}, then one gets that the constant $C$  depends only on $\widetilde{\eta}_X$ and the Banach-Mazur distance from $X$ to the Hilbert space.

\begin{corollary}\label{cor:Hilbert-space-universal-domain}
Let $X$ be a universal pointwise BPB domain space. If $X$ is isomorphic to a Hilbert space, then $X$ has an optimal modulus of convexity. That is,  there exists $C>0$ such that
$$
\delta_X(\eps)\geq C\,\eps^2 \qquad (0<\eps<2).
$$
Moreover, the constant $C$  depends only on $\widetilde{\eta}_X$ and the Banach-Mazur distance from $X$ to the Hilbert space.
\end{corollary}

The main consequence of the above particular case is the following somehow surprising result.

\begin{corollary}\label{cor1}
Let $2<p<\infty$ and let $\mu$ be a positive measure. If the dimension of  $L_p(\mu)$ is greater than or equal to $2$, then $L_p(\mu)$ is not a universal pointwise BPB domain space.
\end{corollary}

\begin{proof}
Suppose that $L_p(\mu)$ is a universal pointwise BPB domain space with dimension bigger than $1$. As there is an one-complemented subspace of $L_p(\mu)$ which is isometric to $\ell_p^2$, Corollary~\ref{corollary-universalpointwise BPB property-one-complemented} shows that $\ell_p^2$ is a universal pointwise BPB domain space. As $\ell_p^2$ is isomorphic to a Hilbert space, Corollary \ref{cor:Hilbert-space-universal-domain} provides that
\[
\delta_{\ell_p^2}(\eps)\geq C\eps^2.
\]
However, $\delta_{\ell_p^2}(\eps) \sim  \eps^p$ and this leads to $p\leq 2$.
\end{proof}

\section{Universal pointwise BPB range spaces}\label{sec:range-spaces}

We start with a sufficient condition for a Banach space $Y$ to be a universal pointwise BPB range space which generalizes the previously known results from \cite{DKL}. We need the following definition which was introduced in \cite{CGKS} to study the BPBp (for Asplund operators). We write $\aconv(A)$ to denote the absolutely convex hull of the set $A$. A subset $\Gamma$ of the unit ball of the dual of a Banach space $X$ is said to be \emph{$1$-norming} (for $X$) if the weak-star closure of $\aconv(\Gamma)$ is the whole of $B_{X^*}$.

\begin{definition}\cite[Definition 3.1]{CGKS} \label{def-ack-str} We say that a Banach space $Y$ has \emph{ACK-structure with parameter $\rho$}, for some $\rho \in [0, 1)$ (ACK$_\rho$-structure, for short), whenever there is a $1$-norming set $\Gamma \subset B_{Y^*}$ such that for every $\eps' > 0$ and every non-empty relatively $w^*$-open subset $U \subset \Gamma$, there exist a non-empty subset $V \subset U$, vectors $y_1^* \in V$, $e \in S_Y$, and an operator $F \in \mathcal{L}(Y,Y)$ with the following properties:
	\begin{itemize}
		\item[(i)] $\|F(e)\| = \|F\| = 1$;
		\item[(ii)] $y_1^*(F(e)) = 1$;
		\item[(iii)] $F^*(y_1^*) = y_1^*$;
		\item[(iv)] $|v^*(F(e))| \leq \rho$ for every $y^* \in \Gamma \setminus V_1$, where $$V_1 = \bigl\{y^* \in \Gamma\colon \|F^*(y^*)\| + (1 -  \eps') \|[\Id_{Y^*} - F^*](y^*)\| \leq 1 \bigr\};$$
		\item[(v)] $\dist \bigl(F^*(y^*),\, \aconv\{0, V\}\bigr) <  \eps'$ for every $y^* \in \Gamma$;
		\item[(vi)] $|v^*(e) - 1| \leq  \eps'$ for every $v^* \in V$.	
	\end{itemize}
\end{definition}

The promised result about the relation between ACK$_\rho$-structure and the pointwise BPB property is the following one, which is analogous to \cite[Theorem 3.4]{CGKS} for the BPBp. Actually, our proof is an adaptation of the proof of that result and that is why we just explain the idea of the proof and the necessary modifications comparing to \cite[Theorem 3.4]{CGKS}.

\begin{proposition}\label{pointwise BPB property:ACK} Every Banach space with ACK$_{\rho}$-structure is a universal pointwise BPB range space.
\end{proposition}

\begin{proof} Let $X$ be a uniformly smooth Banach space, and $Y$ be a Banach space enjoying the ACK$_{\rho}$-structure with respect to the corresponding $1$-norming set $\Gamma \subset B_{Y^*}$. Since $X$ is uniformly smooth, it is reflexive and every operator from $X$ into $Y$ is Asplund.

Arguing in the same way as in \cite[Lemma 2.3]{cas-alt23} with the only difference that instead of the Bishop-Phelps-Bollob\'as theorem one has to apply the fact that the pair $(X, \K)$ possesses the pointwise BPB property (recall that $X$ is uniformly smooth!), we can demonstrate the following: there is a function $\eta: (0, 1) \longrightarrow (0, 1)$ such that for every $T \in \mathcal{L}(X,Y)$ with $\norm{T}=1$,  every $\eps \in (0, 1)$, every $x_0\in S_X$ such that $\norm{Tx_0} > 1-\eta(\eps)$ and every $r>0$ there exists
 a $w^*$-open set $U_r \subset Y^*$ with $U_r \cap \Gamma \neq\emptyset$,
   such that
\begin{equation} \label{eqT^*v}
\norm{T^*z^* - J_X(x_0)}\leq  r +  \eps
\end{equation}
for all $z^*\in U_r \cap \Gamma$.

Since $U_r \cap \Gamma \neq\emptyset$, we can apply Definition \ref{def-ack-str} to $U =  U_r  \cap \Gamma$ and $\eps' > 0$ and obtain a {non-empty} $V  \subset U$, $y_1^* \in V$, $e \in S_Y$, $F \in \mathcal{L}(Y,Y)$ and $V_1 \subset \Gamma$  which satisfy properties (i) -- (vi). In particular, for every $z^* \in V \subset U_r \cap \Gamma$ the inequality \eqref{eqT^*v} holds true. Arguing in the same way as in the proof of \cite[Lemma~3.5]{CGKS}, one can specify the values of $r$ and $\eps'$ and select $\widetilde{\eps} \in [\eps', 1)$  in such a way that the formula
\begin{equation*}\label{eq:operadorDef-new}
S(x) = \langle J_X(x_0), x \rangle Fe+(1-\widetilde{\eps})[I_Y-F](Tx)
\end{equation*}
defines a norm-one operator $S\colon X\longrightarrow Y$ that satisfies $\|S(x_0)\| = 1$ and $\|T - S\| < \eps\left( 2 + \frac{2}{1 - \rho + \eps}\right)$. Taking into account that $\eps$ is arbitrarily small, this gives us what we want.
\end{proof}

Let us comment, for further use, that when one starts in the above proof with $T\in \mathcal{K}(X,Y)$, then the operator $S$ also belongs to $\mathcal{K}(X,Y)$.

One important remark here is that given a Banach space $Y$ with ACK$_\rho$-structure, the function $\tilde\eta(X,Y)$ depends on the Banach space $X$ (see Corollary~\ref{corollary:no-universalfunctionforrange} and the paragraph before it). In the above proof this dependence appears when we introduce the function $\eta$. The dependence on the value of $\rho$ is the only form of dependence of $\tilde\eta(X,Y)$ on $Y$ in the above proof.

As the main consequence of Proposition \ref{pointwise BPB property:ACK}, we get the following list of Banach spaces which are universal pointwise BPB ranges spaces by taking into account the examples and stability results for the ACK$_\rho$-structure given in \cite[\S 4]{CGKS}.

\begin{corollary}\label{corollary-list-ACK-rho}
The following Banach spaces have ACK$_\rho$-structure and, therefore, are universal pointwise BPB range spaces:
\begin{enumerate}
\item[(a)] $C(K)$ and $C_0(L)$ spaces and, more generally, uniform algebras;
\item[(b)] Banach spaces with property $\beta$;
\item[(c)] finite $\ell_\infty$-sums of Banach spaces with ACK$_\rho$-structure;
\item[(d)] finite injective tensor products of spaces with ACK$_\rho$-structure;
\item[(e)] $c_0(Y)$, $\ell_\infty(Y)$ and $c_0(Y,w)$ (the space of weakly-null sequences of $Y$) when $Y$ has ACK$_\rho$-structure;
\item[(f)] $C(K,Y)$ and $C_w(K,Y)$ when $Y$ has ACK$_\rho$-structure.
\end{enumerate}
\end{corollary}

Our next aim is to present a huge number of examples of Banach spaces which are not universal pointwise BPB range spaces. The first result is the more general one.

\begin{theorem}\label{theorem-conv-smooth-not-universal-range}
Every simultaneously uniformly convex and uniformly smooth Banach space $Y$ of dimension greater than  $1$ is not a universal pointwise BPB range space.
\end{theorem}

We need some technical lemmas prior to give the proof of the theorem.

\begin{lemma}[{\cite[Proposition 3 and Corollary 5]{Fig76}}] \label{lemma-modulus-properties}
Let $X$ be a Banach space. Then its modulus of convexity $\delta_X$ has the following property:  both $\delta_X(t)$ and  $\delta_X(t)/t$ are non-decreasing functions.
\end{lemma}

The next one is extracted from the proof of \cite[Theorem 1]{Gur65}.

\begin{lemma} \label{lemma-distances}
Let $X$ be a normed space, $x, y \in S_X$, $t > 0$. Then
$$
\|x - ty\| \geq \frac12 \|x - y\|.
$$
\end{lemma}

\begin{proof}
At first,
$$
\|y - ty\| = |1 - t|= \left| \|x\| - \|ty\| \right| \leq \|x - ty\|.
$$
Consequently,
\[
\|x - y\| \leq \|x - ty\| + \|y - ty\| \leq 2  \|x - ty\|.\qedhere
\]
\end{proof}

The last of this series of lemmas is maybe also known, but we were not able to find it in the literature.

\begin{lemma} \label{lemma-sum-of-norms}
Let $X$ be a linear space, $p_1, p_2$ be two norms on $X$ with
$$
p_2 \leq p_1
$$
and let $(X, p_1)$ be uniformly convex with modulus of convexity $\delta_1$. Denote $\|x\| = p_1(x) + p_2(x)$, then for every $t \in (0, 1)$ the modulus of convexity $\delta$ of $(X, \|\cdot\|)$ satisfies
$$
\delta(t) \geq \delta_1( t/ 4).
$$
\end{lemma}
\begin{proof}
At first, remark that $\frac12 \|\cdot\| \leq p_1(\cdot) \leq  \|\cdot\|$. In order to estimate  $\delta(t)$, fix $x, y \in S_{(X, \|\cdot\|)}$ with $\|x - y\| = t < 1$.  Without loss of generality we may assume that $1/2 \leq p_1(y) \leq p_1(x) \leq 1$. Denote $u = \frac{x}{p_1(x)}$, $v = \frac{y}{p_1(y)}$, $\tau = \frac{p_1(y)}{p_1(x)}$. We have that
\begin{align*}
\frac{p_1(x+y) }{ p_1(x)} & =  p_1\left(\frac{x}{ p_1(x)}+\frac{y}{ p_1(x)}\right)  \\
& = p_1\left(u+ \tau v\right) =p_1\left((1 - \tau)u+ \tau (u + v)\right) \leq  (1 - \tau) + \tau p_1(u + v) \\
& \leq (1 - \tau) + 2 \tau (1 - \delta_1(p_1(u - v))) = 1 + \tau -2 \tau \delta_1(p_1(u - v)),
\end{align*}
that is
$$
p_1(x+y) \leq p_1(x) + p_1(y)  - 2 p_1(y)  \delta_1(p_1(u - v)).
$$
Lemma \ref{lemma-modulus-properties} implies that $p_1(y)  \delta_1(p_1(u - v)) \geq \delta_1(p_1(y) p_1(u - v))$, and $\|\tau x - y\| \geq  \|x -  y\| / 2$ by Lemma \ref{lemma-distances}, so we can continue as follows:
\begin{align*}
p_1(x+y) & \leq p_1(x) + p_1(y) - 2 \delta_1(p_1(y) p_1(u - v)) =  p_1(x) + p_1(y) - 2 \delta_1( p_1(\tau x -  y)) \\
&   \leq p_1(x) + p_1(y) - 2 \delta_1( \|\tau x - y\| / 2)  \leq p_1(x) + p_1(y) - 2 \delta_1( \|x -  y\| / 4) \\
& = p_1(x) + p_1(y) - 2 \delta_1( t/ 4).
\end{align*}
It remains to remark that
\begin{align*}
\|x + y\|  &= p_1(x+y) + p_2(x+y) \leq  p_1(x) + p_1(y) - 2 \delta_1( t/ 4)+ p_2(x) +  p_2(y) \\
& = \|x\| + \|y\| - 2 \delta_1( t/ 4) = 2 -  2 \delta_1( t/ 4).\qedhere
\end{align*}
\end{proof}

We are now able to give the proof of the enounced result.

\begin{proof}[Proof of Theorem \ref{theorem-conv-smooth-not-universal-range}]
Consider on $Y$ a greater equivalent norm $p$ such that $(Y, p)$ is uniformly smooth but not strictly convex. This can be done as follows: represent $Y$ as the direct sum $Y = Y_1\oplus Y_2$ with $Y_2$ being of dimension $2$, introduce a smooth but not strictly convex norm $q$ on $Y_2$, and define
$$
p(y_1 + y_2) = \alpha \sqrt{ \|y_1\|^2 + q(y_2)^2}
$$
for every $y_1 \in Y_1$, $y_2 \in Y_2$, with $\alpha > 0$ big enough.
Denote by $U$ the unit ball of $(Y, p)$.  For every $n\in \N$, let
$B_n = B_{Y} + \frac{1}{n}U$, and introduce the norm $\|\cdot\|_n$ on $Y$ whose unit ball equals $B_n$. We claim that the spaces $X_n =(Y, \|\cdot\|_n)$ have the following properties:
\begin{enumerate}[(i)]
\item  the identity operator $\Id_n : X_n\longrightarrow Y$ satisfies $\|x\|_n \leq \|\Id_n(x)\|$ for all $x \in X_n$, and $\lim\limits_{n\to \infty} \|\Id_n\|=1$;
\item $X_n$ is not strictly convex;
\item the moduli of convexity of all the spaces $X_n^*$ are bounded below by the same positive function $\delta_{Y^*}(t/4)$.
\end{enumerate}

Indeed, (i) is evident. Since $Y$ is reflexive, in order to demonstrate  (ii) it is sufficient to demonstrate that $X_n^*$ is not smooth. Recall that the norm on $X_n^*$ can be expressed as \begin{equation} \label{eq:dualtosumballs}
\|f\|_n^* = \|f\|_{Y^*} + \frac{1}{n}\|f\|_{(Y, p)^*}
\end{equation}
(see, for example, \cite[Lemma 2.1]{KLMW2018}).
The first term in this expression is smooth but the second is not, so the sum is not smooth. At last (iii), in view of formula \eqref{eq:dualtosumballs}, is a direct consequence of Lemma \ref{lemma-sum-of-norms}.

Finally, define
$$
X=\Bigl[\bigoplus_{n=1}^\infty X_n\Bigr]_{\ell_2}.
$$
The corresponding dual space $X^*$ is uniformly convex because of (iii) and, consequently, $X$ is uniformly smooth. It remains to demonstrate that the pair $(X, Y)$ does not have the pointwise BPB property.  Suppose, for the sake of contradiction, that  $(X, Y)$ possesses the pointwise BPB property with a function $\tilde\eta$. Since each space $X_n$ is one-complemented in $X$, this implies (see Proposition \ref{prop:abssumdomain:pointwise BPB property}) that all the pairs $(X_n, Y)$ have the pointwise BPB property with the same function $\tilde\eta(\eps)$ for every $n\in \N$.

Applying (i), we can choose $m\in \N$ such that
$$
\bigl| \|\Id_m\|-1\bigr|<\frac18 \quad \text{and} \quad \frac{1}{\|\Id_m\|} > 1 -\tilde\eta\left(\frac{1}{8}\right).
$$
Write $T_m=\frac{\Id_m}{\|\Id_m\|}$ and observe that for every $x\in S_{X_m}$
$$
\|T_{m}x\| \geq \frac{\|x\|}{\|\Id_m\|}\geq \frac{1}{\|\Id_m\|} > 1 -\tilde\eta\left(\frac{1}{8}\right).
$$
From the assumption,  for every $x \in S_{X_m}$ there exists $F_x \colon X_m\longrightarrow Y$ such that
$$
\|T_m-F_x\|<\frac{1}{8} \qquad \text{and} \qquad \|F_x x\|=\|F_x\|=1.
$$
It follows that $\|\Id_m - F_x\|<1/4$ and so $\|F_x^{-1}\|<4/3$ by Lemma \ref{lemma:norm_of_the_inverse}. As in the last part of the proof of Theorem \ref{thm1}, we get that for given $\eps\in (0,1)$, if $x,z\in S_{X_m}$ satisfy $\re \langle J_{X_m}(x), z \rangle >1-\delta_{Y}\left(\frac{3}{8}\eps\right)$ then $\|x-z\|<\frac{\eps}{2}$.
Indeed, for $y^*\in Y^*$ such that $y^*(F_x x)=1=\|y^*\|$, we have that $J_{X_m}(x)=F_x^*y^*$ by smoothness, and so
$$
\re y^*(F_x z)=\re \langle J_{X_m}(x), z \rangle >1-\delta_{Y}\left(\frac{3}{8}\eps\right).
$$
This shows that $\|F_x x - F_x z\|<\frac{3}{8}\eps$ and so $\|x-z\|_m<\frac{1}{2}\eps$. Finally, Lemma \ref{lemma-estimation-modulus-of-convexity} gives us that  $$\delta_{X_m}(\eps)\geq \frac{1}{2}\delta_{Y}\left(\frac{3}{8}\eps\right),$$ which gives the desired contradiction because the space $X_n$ is not even strictly convex by (ii).
\end{proof}

As important particular cases, we get the following.

\begin{corollary}\label{corollary-ell_p-is-not-universal-pointwise BPB property-range}
For $p\in (1, +\infty)$, the spaces $L_p(\mu)$ of dimension greater than $1$ (in particular, $\ell_p$ and $\ell_p^n$ for $n\geq 2$) fail to be universal pointwise BPB range spaces.
\end{corollary}

In the case of $Y =  \ell_p^2$ for $p\geq 2$, a little modification of the proof of Theorem \ref{theorem-conv-smooth-not-universal-range} gives more.

\begin{theorem}\label{theorem-ellp2-not-universal-range}
For each $p\in [2, +\infty)$  there is a uniformly convex and uniformly smooth Banach space $\widetilde{X}_p$ such that $(\widetilde{X}_p,\ell_p^2)$ fails the pointwise BPB property.
\end{theorem}

\begin{proof}
Set $X_n=\ell_{p+\frac{1}{n}}^2$ for every $n\in \N$. We define
$$
\widetilde{X}_p=\Bigl[\bigoplus_{n=1}^\infty X_n\Bigr]_{\ell_2}
$$
and we observe that $\widetilde{X}_p$ is both uniformly convex and uniformly smooth since the moduli of convexity of all the spaces $X_n$ and $X_n^*$ are bounded below by some positive function. As in the proof of Theorem \ref{theorem-conv-smooth-not-universal-range}, if  for the sake of contradiction we assume that the pair $(\widetilde{X}_p,\ell_p^2)$ has the pointwise BPB property with a function $\tilde\eta$ this will imply for some $m \in \N$ the inequality  $\delta_{\ell_{p+\frac{1}{m}}^2}(\eps)\geq \frac{1}{2}\delta_{\ell_p^2}\left(\frac{3}{8}\eps\right)$, which is not true.
\end{proof}

Let us observe that the pair $(\widetilde{X}_p,\ell_p^2)$ has the BPBp as $\widetilde{X}_p$ is uniformly convex (see \cite[Theorem 2.2]{ABGM} or \cite[Theorem 3.1]{KL}).

As a consequence of Theorem  \ref{theorem-conv-smooth-not-universal-range}, we obtain also that the pointwise BPB property is not stable under finite $\ell_p$-sums of the range space for $1 < p< \infty$: just consider the pairs $(X, \R)$, $(X, \R)$ and $(X, \R \oplus_p \R) = (X, \ell_p^2)$, where $X$ is such a uniformly smooth space that  the pair $(X, \ell_p^2)$ fails to have the pointwise BPB property.  Let us recall that it follows as a particular case of \cite[Proposition 2.9]{DKL} that the pointwise BPB property is stable under finite $\ell_\infty$-sums of the range space.

Our next example will be used to show that the pointwise BPB property is not stable under infinite $c_0$- or $\ell_p$-sums ($1\leq  p \leq \infty$) of the range spaces. We use similar arguments to the ones given in \cite[Example 4.1]{ACKLM}.

\begin{example}\label{example-pointwise BPB property-range}
{\slshape There exist a uniformly smooth two-dimensional Banach space $X$ and a sequence of polyhedral two-dimensional spaces $\{Y_n\colon n\in \N\}$ such that
\begin{equation*}
	\inf_{n \in \N} \tilde\eta(X, Y_n)(\eps) = 0
\end{equation*}
for $0<\eps<1/4$.}
\end{example}

\begin{proof}
Consider the convex set $C=A\cup \left( \bigcup_{(i,j)\in \{1,2\}\times \{1,2\}} A_{i,j}\right) \subset \R^2$, where
\begin{equation*}
A=\left\{ (x,y)\in \R^2\colon \max\{ |x|, |y|\} \leq 1 \  \text{and} \   \min \{ |x|, |y| \} \leq 1 - \frac{1}{8} \right\}
\end{equation*}
and
\begin{equation*}
A_{i,j}= \left\{ (x,y)\in \R^2\colon \left( x - (-1)^i \left(1 - \frac{1}{8} \right) \right)^2 + \left( y - (-1)^j \left(1 - \frac{1}{8} \right) \right)^2 \leq \frac{1}{8^2} \right\}.
\end{equation*}

Note that $C$ is the square with rounded corners and the Minkowski functional $\mu_C$ is a norm in $\R^2$. We consider the real Banach space $X = (\R^2, \mu_C)$. So, $C = B_X$ and we observe that $X$ is uniformly smooth by construction.

Next, for every $n\in \N$, let $Y_n$ be $\R^2$ endowed with the norm
\begin{equation*}
\|(x, y)\|_n = \max \left\{ |x|, |y| + \frac{1}{n}|x| \right\} \qquad \bigl((x, y) \in \R^2\bigr).
\end{equation*}
Let $w = \left(1 - \frac{1}{8}, 1\right)$, $v =\left(-1 + \frac{1}{8}, 1\right)$, and $e_2=(0,1)$. Then
\begin{equation*}
\|w\|=\|v\|=\|e_2\|=1, \qquad\ \	e_2 = \frac{w + v}{2}, \qquad \text{and} \qquad w - v =  \left(\frac{7}{4},0\right).
\end{equation*}
Fix $\eps\in (0,1/4)$. We claim that $\inf_{n \in \N} \tilde\eta(X, Y_n)(\eps) = 0$. Suppose that it is not the case and take $\tilde\eta(\eps) > 0$ such that $\inf_{n \in \N} \tilde\eta(X, Y_n) > \tilde\eta(\eps) > 0$. Consider $\Id_n: X \longrightarrow Y_n$ to be the identity operator from $X$ into $Y_n$ and observe that
$\lim_{n\to \infty}\|\Id_n\|=1$. Therefore, we may choose $m\in \N$ such that
\begin{equation*}
\bigl|\|\Id_m\|-1\bigr|<\eps/2 \qquad \text{and} \qquad \frac{\|\Id_{m}(e_2)\|_n}{\|\Id_m\|}=\frac{1}{\|\Id_m\|} > 1 - \tilde\eta(\eps/2).
\end{equation*}
Next, defining $T_m:=\Id_m/\|\Id_m\|$, we get that $\|T_m\|=1$ and
$$
\|T_m(e_2)\|>1-\tilde\eta(\eps/2),
$$
so there is $S_m\in \mathcal{L}(X, Y_{m})$ such that

\begin{equation*}
\|S_m\| = \|S_m(e_2)\|_m = 1 \quad \text{and} \quad \left\|T_{m} - S_m \right\| < \frac{\eps}{2}.
\end{equation*}
Therefore,
$$
\|\Id_m-S_m\|\leq \|\Id_m - T_m\| + \|T_m-S_m\|\leq \bigl|\|\Id_m\|-1\bigr| + \frac{\eps}{2}<\eps <\frac{1}{4}.
$$
Now, we have that
\begin{equation*}
1 = \|S_m(e_2)\|_m \leq \frac{1}{2} \|S_m(w)\|_m + \frac{1}{2}\|S_m(v) \|_m \leq 1,
\end{equation*}
which implies that $\|S_m(e_2)\|_m = \|S_m(w)\|_m = \|S_m(v)\|_m = 1$. This shows that the line segment $[S_m(w), S_m(v)]$ lies on the unit sphere of $Y_m$. Since it contains $S_m(e_2)$ and the absolute value of the first coordinate of this vector can not be equal to one (indeed, $\|S_m(e_2)-e_2\|_m\leq \|S_m-\Id_m\|<\eps<1/4$), it follows that $\|S_m(w) - S_m(v)\|_m \leq 1$ from the shape of $B_{Y_m}$. With this in mind, we have that
\begin{equation*}
\|\Id_{m}(w) - S_m(v)\|_m \leq \|\Id_{m} (w) - S_m(w)\|_m + \|S_m(w) - S_m(v)\|_m < \eps + 1 < \frac{5}{4}.
\end{equation*}
On the other hand, since $\|\Id_{m}(w) - \Id_{m}(v)\|_m = \|w - v\| = \frac{7}{4}$, we have that
\begin{equation*}
\|\Id_{m}(w) - S_m(v)\| \geq \|\Id_{m}(w) - \Id_{m}(v)\| - \|\Id_{m}(v) - S_m(v)\|_m > \frac{7}{4} - \eps > \frac{5}{4}.	
\end{equation*}
This gives the desired contradiction.
\end{proof}

It follows that the pointwise BPB property is not stable under infinite $c_0$- or $\ell_p$-sums of the range spaces.

\begin{corollary}
The pointwise BPB property is not stable under infinite $c_0$- or $\ell_p$-sums of the range space  for $1\leq p \leq \infty$. Moreover, being a universal pointwise BPB range space is also not stable under infinite $c_0$- or $\ell_p$-sums with $1\leq p \leq \infty$.
\end{corollary}

\begin{proof}
Let $X$ and $\{Y_n\colon n\in \N\}$ be spaces given in Example \ref{example-pointwise BPB property-range} and consider the spaces
$$
\mathcal{Y}_0=\Bigl[\bigoplus_{n=1}^\infty Y_n\Bigr]_{c_0}, \qquad \mathcal{Y}_p=\Bigl[\bigoplus_{n=1}^\infty Y_n\Bigr]_{\ell_p} \qquad (1\leq p \leq \infty).
$$
As the spaces $Y_k$ are polyhedral, they have property $\beta$ and so they are universal pointwise BPB range spaces by Corollary \ref{corollary-list-ACK-rho} (or by \cite[Proposition 2.4]{DKL}). But $\inf\limits_n\tilde\eta(X,Y_n)(\eps)=0$, so it follows that $(X,\mathcal{Y}_0)$ and $(X,\mathcal{Y}_p)$ fail the pointwise BPB property by Corollary \ref{corollary:sums-range}.
\end{proof}

It also follows that property quasi-$\beta$, a weakening of property $\beta$ introduced in \cite{AAP} which still implies density of norm attaining operators, is not enough to be a universal pointwise BPB range space.

\begin{example}
{\slshape Consider the sequence of spaces $\{Y_n\colon n\in \N\}$ defined in Example \ref{example-pointwise BPB property-range}. Then the space $\Bigl[\bigoplus_{n=1}^\infty Y_n\Bigr]_{c_0}$ has property quasi-$\beta$ \cite[Proposition 4]{AAP} but it fails to be a universal pointwise BPB range space as showed in the previous proof.}
\end{example}

\section{The pointwise BPB property for compact operators}\label{sect:compact}

In this section we deal with the corresponding pointwise BPB property property for compact operators. Let us start with the needed definitions.

\begin{definition} \label{pointwise BPB property:comp:def}
We say that the pair $(X, Y)$ of Banach spaces has the \emph{pointwise BPB property for compact operators} if given $\eps > 0$, there is $\tilde\eta(\eps) > 0$ such that whenever $T \in \mathcal{K}(X, Y)$ with $\|T\| = 1$ and $x_0 \in S_X$ satisfy
	\begin{equation*}
	\|T(x_0)\| > 1 - \tilde\eta(\eps),	
	\end{equation*}
there is $S \in \mathcal{K}(X, Y)$ with $\|S\| = 1$ such that
\begin{equation*}
\|S(x_0)\| = 1 \qquad \text{and} \qquad \|S - T\| < \eps.
\end{equation*}
A Banach space $X$ is said to be \emph{universal pointwise BPB domain space for compact operators} if $(X,Z)$ has the pointwise BPB property for compact operators for every Banach space $Z$. A Banach space $Y$ is said to be \emph{universal pointwise BPB range space for compact operators} if $(Z,Y)$ has the pointwise BPB property for compact operators for every uniformly smooth Banach space $Z$.
\end{definition}

In the same way that we may do with the pointwise BPB property, we may consider operators $T$ with norm less than or equal to one in Definition \ref{pointwise BPB property:comp:def}.

Many of the results given in the previous sections for the pointwise BPB property can be adapted to the pointwise BPB property for compact operators, as in many proofs, when one starts with compact operators, then all the involved operators are also compact. Let us summarize all these results in the next proposition.

\begin{proposition} \label{pointwise BPB property-K:many-results} Let $X$ and $Y$ be Banach spaces. 	
\begin{enumerate}
\item[(a)] Hilbert spaces are universal pointwise BPB domain spaces for compact operators.
\item[(b)] If $(X, Y)$ has the pointwise BPB property for compact operators, then $X$ is uniformly smooth.
\item[(c)] If $X_1$ is one-complemented in $X$ and the pair $(X,Y)$ has the pointwise BPB property for compact operators, then so does $(X_1,Y)$ with the same function.
\item[(d)] If $Y=Y_1 \oplus_a Y_2$ for some absolute sum $\oplus_a$ and $(X,Y)$ has the pointwise BPB property for compact operators with a function $\eps\longmapsto \tilde\eta(\eps)$, then so does $(X,Y_1)$ with the function $\eps \longmapsto \tilde\eta(\eps/3)$. If $\oplus_a=\oplus_1$ or $\oplus_a=\oplus_\infty$, then there is no need to divide $\eps$ by $3$.
\item[(e)] If $K$ is a compact Hausdorff topological space and the pair $(X,C(K,Y))$ has the pointwise BPB property for compact operators, then so does $(X,Y)$ with the same function.
\item[(f)] If $X$ is a universal pointwise BPB domain space for compact operators, then there is a common function $\tilde\eta$ such that for every Banach space $Z$ the pair $(X,Z)$ has the pointwise BPB property for compact operators with the function $\tilde\eta$.
\item[(g)] The space $L_p(\mu)$ fails to be universal pointwise BPB domain space for compact operators whenever $2<p<\infty$ and $\dim\bigl(L_p(\mu)\bigr)\geq 2$.
\item[(h)] Every Banach space with ACK$_\rho$-structure is a universal pointwise BPB range space for compact operators. In particular, the following spaces have ACK$_\rho$-structure and so they are universal pointwise BPB range spaces for compact operators:
    \begin{itemize}
\item $C(K)$ and $C_0(L)$ spaces and, more in general, uniform algebras;
\item Banach spaces with property $\beta$;
\item finite $\ell_\infty$-sums of Banach spaces with ACK$_\rho$-structure;
\item finite injective tensor products of spaces with ACK$_\rho$-structure;
\item $c_0(Y)$, $\ell_\infty(Y)$ and $c_0(Y,w)$ (the space of weakly-null sequences of $Y$) when $Y$ has ACK$_\rho$-structure;
\item $C(K,Y)$ and $C_w(K,Y)$ when $Y$ has ACK$_\rho$-structure.
\end{itemize}
\item[(i)] For $1< p<\infty$, the space $\ell_p$ is not a universal pointwise BPB range space for compact operators. This also happens for every uniformly convex and uniformly smooth space containing a finite-dimensional absolute summand.
\end{enumerate}
\end{proposition}

\begin{proof}
(a). We can adapt the proof of \cite[Theorem 2.5]{DKL} to compact operators, using \cite[Example 1.5.(b)]{DGMM} to get that all the pairs $(H,Y)$ have the BPBp for compact operators when $H$ is a Hilbert space and $Y$ is arbitrary.

(b). In the proof of \cite[Proposition 2.3]{DKL}, only rank-one (hence compact) operators are used.

(c), (d), and (e). We may just adapt the proof of Proposition \ref{prop:abssumdomain:pointwise BPB property}, Proposition \ref{prop:abssumrange:pointwise BPB property}, and Proposition \ref{prop:X-CKY}, respectively, to compact operators.

(f). We can adapt the proof of Corollary \ref{corollary:universal-eta} to compact operators using item (d) above.

(g). The proof of Corollary \ref{cor1} can be adapted to compact operators using item (c) above instead of Corollary \ref{corollary-universalpointwise BPB property-one-complemented}.

(h). At the end of the proof of Proposition \ref{pointwise BPB property:ACK} it is commented that everything works for compact operators.

(i). If $\ell_p$ is a universal pointwise BPB range space for compact operators, then so does $\ell_p^2$ by item (d) above. But then $\ell_p^2$ would be a universal pointwise BPB range space, contradicting Corollary \ref{corollary-ell_p-is-not-universal-pointwise BPB property-range}. The same proof works for any Banach space which is uniformly convex and uniformly smooth and contains a finite-dimensional absolute summand, using Theorem \ref{theorem-conv-smooth-not-universal-range} instead of Corollary \ref{corollary-ell_p-is-not-universal-pointwise BPB property-range}.
\end{proof}

Our next aim is to provide some results for the pointwise BPB property for compact operators which do not come from the results for general operators. In \cite{DGMM} some results for the BPBp for compact operators were given both on domain and range spaces. We do not know whether the analogous results for the pointwise BPB property on domain spaces hold or not, but the results on range spaces do. The main tool is the following result.

\begin{proposition}\label{pointwise BPB property:comp:range} Let $X$ be a uniformly smooth Banach space and let $Y$ be an arbitrary Banach space. Suppose that there exists a net of norm-one projections $\{Q_{\lambda}\}_{\lambda\in \Lambda} \subset \mathcal{L}(Y,Y)$ such that $\{Q_{\lambda} y\} \longrightarrow y$ in norm for every $y\in Y$. If there is a function $\tilde\eta:\R^+\longrightarrow \R^+$ such that all the pairs $(X, Q_{\lambda} (Y))$ with $\lambda\in \Lambda$ have the pointwise BPB property for compact operators with the function $\tilde\eta$, then the pair $(X, Y)$ has the pointwise BPB property for compact operators.
\end{proposition}

The proof is a routine adaptation to the pointwise BPB property case of the one given in \cite[Proposition 2.5]{DGMM} for the BPBp for compact operators.

A first consequence of this proposition is that Lindenstrauss spaces (i.e.\ isometric preduals of $L_1(\mu)$ spaces) are universal pointwise BPB range spaces for compact operators. This is so by just taking into account the classical result by Lazar and Lindenstrauss that every finite subset of a Lindenstrauss space is contained in a subspace of it which is isometrically isomorphic to an $\ell_\infty^n$ space. Now, all these spaces are one-complemented and they are universal pointwise BPB range spaces; moreover, for a fixed Banach space $X$, the functions $\tilde\eta(X,\ell_\infty^n)$ do not depend on $n$. Alternatively, it is possible to adapt to the pointwise BPB property for compact operators the analogous result for the BPBp for compact operators given in \cite[Theorem 4.2]{8authors}.

\begin{corollary}
Let $Y$ be a Banach space such that $Y^*$ is isometrically isomorphic to some $L_1(\mu)$ space. Then $Y$ is a universal pointwise BPB range space for compact operators.
\end{corollary}

We next may use Proposition \ref{pointwise BPB property:comp:range} to get some results for vector-valued function spaces.

\begin{proposition} \label{pointwise BPB property:comp:range2} Let $X$ be a uniformly smooth Banach space and let $Y$ be an arbitrary Banach space.
\begin{itemize}
\item[(a)]	For $1 \leq p < \infty$, if the pair $(X, \ell_p(Y))$ has the pointwise BPB property for compact operators, then so does the pair $(X, L_p(\mu, Y))$ for every positive measure $\mu$.
\item[(b)] If the pair $(X, Y)$ has the pointwise BPB property for compact operators, then so does the pair $(X, L_{\infty}(\mu, Y))$ for every $\sigma$-finite positive measure $\mu$.
\item[(c)] The pair $(X, Y)$ has the pointwise BPB property for compact operators if and only if the pair $(X, C(K, Y))$ has the pointwise BPB property for compact operators for every compact Hausdorff topological space $K$.	
\end{itemize}	
\end{proposition}

The proof is a simple adaptation of \cite[Theorem 3.15]{DGMM}, but let us note that in item (a), the proof given there only covers the case when $L_1(\mu)$ is infinite-dimensional, but the remaining case follows immediately from Corollary \ref{corollary:sums-range}.

\section{Some open problems}

In this last section, we compile some open problems about the contents of this paper.

Taking into account Proposition \ref{prop:abssumdomain:pointwise BPB property}, we do not know if it is possible to get analogous results for the Bishop-Phelps-Bollob\'as property or for the denseness of norm attaining operators.

\begin{problem} Let $X$, $Y$ be Banach spaces and let $X_1$ be a one-complemented subspace of $X$. Is it true that $(X_1, Y)$ has the BPBp whenever $(X, Y)$ satisfies it? Is it true that the set of norm attaining operators from $X_1$ to $Y$ is dense in $\mathcal{L}(X_1,Y)$ whenever the set of norm attaining operators from $X$ to $Y$ is dense in $\mathcal{L}(X,Y)$?
\end{problem}

Still in the sprint of Proposition \ref{prop:abssumdomain:pointwise BPB property}, we also do not know if such a general result holds true for ranges spaces.

\begin{problem} Let $X$ and $Y$ be Banach spaces. Let $Y_1$ be a one-complemented subspace of $Y$. Is it true that $(X, Y_1)$ has the pointwise BPB property whenever $(X, Y)$ satisfies it? It seems to be open also for the BPBp.
\end{problem}

By Corollary \ref{cor1}, we know that $L_p(\mu)$ is not a universal pointwise BPB domain space when $2 < p < \infty$ and $\mu$ is a positive measure such that $\dim(L_p(\mu))\geq 2$. We do not know what happens for $1 < p < 2$.

\begin{problem} Is the space $L_p(\mu)$ a universal pointwise BPB domain space for any $1 < p < 2$?
\end{problem}

As the proof of Corollary \ref{cor1} is not constructive, we do not know a concrete Banach space $Y$ such that the pair $(L_p(\mu), Y)$ fails the pointwise BPB property.

\begin{problem} Find a concrete example of a Banach space $Y$ so that the pair $(L_p(\mu), Y)$ fails the pointwise BPB property for $2 < p < \infty$.
\end{problem}

%

Finally, we do not know if there is any relationship between the pointwise BPB property and the pointwise BPB property for compact operators.

\begin{problem} Does the pointwise BPB property imply the pointwise BPB property for compact operators or viceversa?
\end{problem}

{\bf Acknowledgement.} The authors  are grateful to the anonymous referee for suggestions that improved the exposition.

\newpage


\begin{thebibliography}{99}


\bibitem{AcostaBJM2016} \textsc{M.~D.~Acosta}, The Bishop-Phelps-Bollob\'{a}s property for operators on C(K), \emph{Banach J. Math. Anal.} \textbf{10} (2016), 307--319.

\bibitem{AAP} \textsc{M.~D.~Acosta, F.~J.~Aguirre and R.~Pay\'a}, A new sufficient condition for the denseness of norm attaining operators, \emph{Rocky Mountain J. Math.} \textbf{26} (1996), 407--418.


\bibitem{AAGM} \textsc{M.~D.~Acosta, R.~M.~Aron, D.~Garc\'ia and M.~Maestre}, The Bishop-Phelps-Bollob\'as theorem for operators, \emph{J. Funct. Anal.} {\bf 294} (2008), 2780--2899.

\bibitem{8authors} \textsc{M.~D.~Acosta, J.~Becerra-Guerrero, Y.~S.~Choi, M.~Ciesielski, S.~K.~Kim, H.~J.~Lee, M.~L.~Louren\c{c}o and M.~Mart\'in}, The Bishop-Phelps-Bollob\'as property for operators between spaces of continuous funtions, \emph{Nonlinear Anal.} {\bf 95} (2014), 323--332.

\bibitem{ABGM}\textsc{M.~D.~Acosta, J.~Becerra-Guerrero, D.~Garc\'{\i}a and M.~Maestre}, The Bishop-Phelps-Bollob\'as Theorem for bilinear forms, \emph{Trans. Amer. Math. Soc.} {\bf 365} (2013), 5911--5932.

\bibitem{AMS} \textsc{M.~D.~Acosta, M.~Masty{\l}o and M.~Soleimani-Mourchehkhorti}, The Bishop-Phelps-Bollob\'as and approximate hyperplane series properties, \emph{J. Funct. Anal.} \textbf{274} (2018), no. 9, 2673--2699.

\bibitem{cas-alt23}
\textsc{R.~M. Aron, B.~Cascales and O.~Kozhushkina}, The {B}ishop-{P}helps-{B}ollob\'as theorem and {A}splund operators, \emph{Proc. Amer.  Math. Soc.} \textbf{139} (2011), no.~10, 3553--3560.

\bibitem{ACKLM}  \textsc{R.~Aron, Y.~S.~Choi, S.~K.~Kim, H.~J.~Lee and M.~Mart\'in}, The Bishop-Phelps-Bollob\'{a}s version of Lindenstrauss properties A and B,  \emph{Trans. Amer. Math. Soc.} \textbf{367} (2015), 6085--6101.

\bibitem{CGKS}  \textsc{B.~Cascales, A.~J.~Guirao, V.~Kadets and M.~Soloviova}, $\Gamma$-Flatness and Bishop-Phelps-Bollob\'as type theorems for operators, \emph{J. Funct. Anal.} \textbf{274} (2018), no. 3, 863--888.

\bibitem{Cho-Choi} \textsc{D.~H.~Cho and Y.~S.~Choi}, The Bishop-Phelps-Bollob\'{a}s theorem on bounded closed convex sets, \emph{J. Lond. Math. Soc.} \textbf{93} (2016),  502--518.

\bibitem{CDJM} \textsc{Y.~S.~Choi, S.~Dantas, M.~Jung and M.~Mart\'in}, On the Bishop-Phelps-Bollob\'as property and absolute sums, \emph{preprint} (2018).


\bibitem{DGMM} \textsc{S.~Dantas, D.~Garc\'ia, M.~Maestre and M.~Mart\'in}, The Bishop-Phelps-Bollob\'as property for compact operators, \emph{Canad. J. Math.} \textbf{70} (2018), no. 1, 53--73.

\bibitem{DKL} \textsc{S.~Dantas, S.~K.~Kim and H.~J.~Lee}, The Bishop-Phelps-Bollob\'as point property, \emph{J. Math. Anal. Appl.} \textbf{444} (2016), 1739--1751.

\bibitem{Diestel} \textsc{J.~Diestel} \emph{Geometry of Banach spaces -- selected topics.} Lecture Notes in Mathematics, Vol. 485. Springer-Verlag, 1975.

\bibitem{FHHMPZ} \textsc{M.~Fabian, P.~Habala, P.~H\`ajek, V.~M.~Santaluc\'ia, J.~Pelant and V.~Zizler}, \emph{Functional Analysis and Infinite-Dimensional Geometry}, Springer, 2000.

\bibitem{Fig76} \textsc{T.~Figiel}, On the moduli of convexity and smoothness,  \emph{Studia Math.} \textbf{56} (1976), 121--155.


\bibitem{Gur65}  \textsc{V.I.~Gurarij}, On moduli of convexity and flattening of Banach spaces. (English. Russian original) \emph{Sov. Math., Dokl.} \textbf{6} (1965), 535--539; translation from Dokl. Akad. Nauk SSSR 161, 1003--1006 (1965).


\bibitem{KLMW2018} \textsc{V.~Kadets, G.~L\'opez, M.~Mart\'{\i}n,  and D.~Werner},  Equivalent norms with an extremely nonlineable set of norm attaining functionals, \emph{J. Inst. Math. Jussieu} (to appear), doi:\href{https://doi.org/10.1017/S1474748018000087}{10.1017/S1474748018000087}

 \bibitem{KL} \textsc{S.~K.~Kim and H.~J.~Lee}, Uniform Convexity and Bishop-Phelps-Bollob\'as Property. \emph{Canad. J. Math.} {\bf 66} (2014),  373--386.

\bibitem{Pi} \textsc{G.~Pisier}, Martingales with values in uniformly convex spaces. \emph{Israel J. Math.} \textbf{20} (1975), 326--350.

\end{thebibliography}
\end{document}